\documentclass[12pt]{amsart}
\usepackage{amsmath,amsthm,amsfonts,amssymb,latexsym}

\usepackage{hyperref}

\usepackage{enumerate}
\usepackage{color}
\usepackage[dvipsnames]{xcolor}

\usepackage[shortlabels]{enumitem}

\begin{document}

\theoremstyle{plain}

\newtheorem{thm}{Theorem}

\newtheorem{lem}[thm]{Lemma}
\newtheorem{Problem B}[thm]{Problem B}

\newtheorem{pro}[thm]{Proposition}
\newtheorem{conj}[thm]{Conjecture}
\newtheorem{cor}[thm]{Corollary}
\newtheorem{que}[thm]{Question}
\newtheorem{prob}[thm]{Problem}
\newtheorem{rem}[thm]{Remark}
\newtheorem{defi}[thm]{Definition}

\newtheorem*{thmA}{Theorem A}
\newtheorem*{thmB}{Theorem B}
\newtheorem*{corB}{Corollary B}
\newtheorem*{thmC}{Theorem C}
\newtheorem*{thmD}{Theorem D}
\newtheorem*{thmE}{Theorem E}
 
\newtheorem*{thmAcl}{Theorem A}
\newtheorem*{thmBcl}{Theorem B}
\newcommand{\dd}{\mathrm{d}}

\newcommand{\Maxn}{\operatorname{Max_{\textbf{N}}}}
\newcommand{\Syl}{\operatorname{Syl}}
\newcommand{\Lin}{\operatorname{Lin}}
\newcommand{\U}{\mathbf{U}}
\newcommand{\R}{\mathbf{R}}
\newcommand{\dl}{\operatorname{dl}}
\newcommand{\Con}{\operatorname{Con}}
\newcommand{\cl}{\operatorname{cl}}
\newcommand{\Stab}{\operatorname{Stab}}
\newcommand{\Aut}{\operatorname{Aut}}
\newcommand{\Ker}{\operatorname{Ker}}
\newcommand{\InnDiag}{\operatorname{InnDiag}}
\newcommand{\fl}{\operatorname{fl}}
\newcommand{\Irr}{\operatorname{Irr}}
\newcommand{\FF}{\mathbb{F}}
\newcommand{\EE}{\mathbb{E}}
\newcommand{\normal}{\trianglelefteq}
\newcommand{\sn}{\normal\normal}
\newcommand{\Bl}{\mathrm{Bl}}
\newcommand{\NN}{\mathbb{N}}
\newcommand{\N}{\mathbf{N}}
\newcommand{\bfC}{\mathbf{C}}
\newcommand{\bfO}{\mathbf{O}}
\newcommand{\bfF}{\mathbf{F}}
\def\GGG{{\mathcal G}}
\def\HHH{{\mathcal H}}
\def\HH{{\mathcal H}}
\def\irra#1#2{{\rm Irr}_{#1}(#2)}

\renewcommand{\labelenumi}{\upshape (\roman{enumi})}

\newcommand{\PSL}{\operatorname{PSL}}
\newcommand{\PSU}{\operatorname{PSU}}
\newcommand{\alt}{\operatorname{Alt}}

\providecommand{\V}{\mathrm{V}}
\providecommand{\E}{\mathrm{E}}
\providecommand{\ir}{\mathrm{Irm_{rv}}}
\providecommand{\Irrr}{\mathrm{Irm_{rv}}}
\providecommand{\re}{\mathrm{Re}}

\numberwithin{equation}{section}
\def\irrp#1{{\rm Irr}_{p'}(#1)}

\def\ibrrp#1{{\rm IBr}_{\Bbb R, p'}(#1)}
\def\C{{\mathbb C}}
\def\Q{{\mathbb Q}}
\def\irr#1{{\rm Irr}(#1)}
\def\irrp#1{{\rm Irr}_{p^\prime}(#1)}
\def\irrq#1{{\rm Irr}_{q^\prime}(#1)}
\def \c#1{{\cal #1}}
\def \aut#1{{\rm Aut}(#1)}
\def\cent#1#2{{\bf C}_{#1}(#2)}
\def\norm#1#2{{\bf N}_{#1}(#2)}
\def\zent#1{{\bf Z}(#1)}
\def\syl#1#2{{\rm Syl}_#1(#2)}
\def\normal{\triangleleft\,}
\def\oh#1#2{{\bf O}_{#1}(#2)}
\def\Oh#1#2{{\bf O}^{#1}(#2)}
\def\det#1{{\rm det}(#1)}
\def\gal#1{{\rm Gal}(#1)}
\def\ker#1{{\rm ker}(#1)}
\def\normalm#1#2{{\bf N}_{#1}(#2)}
\def\alt#1{{\rm Alt}(#1)}
\def\iitem#1{\goodbreak\par\noindent{\bf #1}}
   \def \mod#1{\, {\rm mod} \, #1 \, }
\def\sbs{\subseteq}

\def\gc{{\bf GC}}
\def\ngc{{non-{\bf GC}}}
\def\ngcs{{non-{\bf GC}$^*$}}
\newcommand{\notd}{{\!\not{|}}}

\newcommand{\Z}{\mathbf{Z}}
\newcommand{\Out}{{\mathrm {Out}}}
\newcommand{\Mult}{{\mathrm {Mult}}}
\newcommand{\Inn}{{\mathrm {Inn}}}
\newcommand{\IBR}{{\mathrm {IBr}}}
\newcommand{\IBRL}{{\mathrm {IBr}}_{\ell}}
\newcommand{\IBRP}{{\mathrm {IBr}}_{p}}
\newcommand{\cd}{\mathrm{cd}}
\newcommand{\ord}{{\mathrm {ord}}}
\def\id{\mathop{\mathrm{ id}}\nolimits}
\renewcommand{\Im}{{\mathrm {Im}}}
\newcommand{\Ind}{{\mathrm {Ind}}}
\newcommand{\diag}{{\mathrm {diag}}}
\newcommand{\soc}{{\mathrm {soc}}}
\newcommand{\End}{{\mathrm {End}}}
\newcommand{\sol}{{\mathrm {sol}}}
\newcommand{\Hom}{{\mathrm {Hom}}}
\newcommand{\Mor}{{\mathrm {Mor}}}
\newcommand{\Mat}{{\mathrm {Mat}}}
\def\rank{\mathop{\mathrm{ rank}}\nolimits}
\newcommand{\Tr}{{\mathrm {Tr}}}
\newcommand{\tr}{{\mathrm {tr}}}
\newcommand{\Gal}{{\rm Gal}}
\newcommand{\Spec}{{\mathrm {Spec}}}
\newcommand{\ad}{{\mathrm {ad}}}
\newcommand{\Sym}{{\mathrm {Sym}}}
\newcommand{\Char}{{\mathrm {Char}}}
\newcommand{\pr}{{\mathrm {pr}}}
\newcommand{\rad}{{\mathrm {rad}}}
\newcommand{\abel}{{\mathrm {abel}}}
\newcommand{\PGL}{{\mathrm {PGL}}}
\newcommand{\PCSp}{{\mathrm {PCSp}}}
\newcommand{\PGU}{{\mathrm {PGU}}}
\newcommand{\codim}{{\mathrm {codim}}}
\newcommand{\ind}{{\mathrm {ind}}}
\newcommand{\Res}{{\mathrm {Res}}}
\newcommand{\Lie}{{\mathrm {Lie}}}
\newcommand{\Ext}{{\mathrm {Ext}}}
\newcommand{\Alt}{{\mathrm {Alt}}}
\newcommand{\AAA}{{\sf A}}
\newcommand{\SSS}{{\sf S}}
\newcommand{\DDD}{{\sf D}}
\newcommand{\QQQ}{{\sf Q}}
\newcommand{\CCC}{{\sf C}}
\newcommand{\SL}{{\mathrm {SL}}}
\newcommand{\Sp}{{\mathrm {Sp}}}
\newcommand{\PSp}{{\mathrm {PSp}}}
\newcommand{\SU}{{\mathrm {SU}}}
\newcommand{\GL}{{\mathrm {GL}}}
\newcommand{\GU}{{\mathrm {GU}}}
\newcommand{\Spin}{{\mathrm {Spin}}}
\newcommand{\CC}{{\mathbb C}}
\newcommand{\CB}{{\mathbf C}}
\newcommand{\RR}{{\mathbb R}}
\newcommand{\QQ}{{\mathbb Q}}
\newcommand{\ZZ}{{\mathbb Z}}
\newcommand{\bfN}{{\mathbf N}}
\newcommand{\bfZ}{{\mathbf Z}}
\newcommand{\PP}{{\mathbb P}}
\newcommand{\cG}{{\mathcal G}}
\newcommand{\OO}{\mathcal O}
\newcommand{\cH}{{\mathcal H}}
\newcommand{\cQ}{{\mathcal Q}}
\newcommand{\GA}{{\mathfrak G}}
\newcommand{\cT}{{\mathcal T}}
\newcommand{\cL}{{\mathcal L}}
\newcommand{\IBr}{\mathrm{IBr}}
\newcommand{\cS}{{\mathcal S}}
\newcommand{\cR}{{\mathcal R}}
\newcommand{\GCD}{\GC^{*}}
\newcommand{\TCD}{\TC^{*}}
\newcommand{\FD}{F^{*}}
\newcommand{\GD}{G^{*}}
\newcommand{\HD}{H^{*}}
\newcommand{\GCF}{\GC^{F}}
\newcommand{\TCF}{\TC^{F}}
\newcommand{\PCF}{\PC^{F}}
\newcommand{\GCDF}{(\GC^{*})^{F^{*}}}
\newcommand{\RGTT}{R^{\GC}_{\TC}(\theta)}
\newcommand{\RGTA}{R^{\GC}_{\TC}(1)}
\newcommand{\Om}{\Omega}
\newcommand{\eps}{\epsilon}
\newcommand{\varep}{\varepsilon}
\newcommand{\al}{\alpha}
\newcommand{\chis}{\chi_{s}}
\newcommand{\sigmad}{\sigma^{*}}
\newcommand{\PA}{\boldsymbol{\alpha}}
\newcommand{\gam}{\gamma}
\newcommand{\lam}{\lambda}
\newcommand{\la}{\langle}
\newcommand{\genf}{F^*}
\newcommand{\ra}{\rangle}
\newcommand{\hs}{\hat{s}}
\newcommand{\htt}{\hat{t}}
\newcommand{\tG}{\hat G}
\newcommand{\St}{\mathsf {St}}
\newcommand{\bfs}{\boldsymbol{s}}
\newcommand{\bfl}{\boldsymbol{\lambda}}
\newcommand{\tn}{\hspace{0.5mm}^{t}\hspace*{-0.2mm}}
\newcommand{\ta}{\hspace{0.5mm}^{2}\hspace*{-0.2mm}}
\newcommand{\tb}{\hspace{0.5mm}^{3}\hspace*{-0.2mm}}
\def\skipa{\vspace{-1.5mm} & \vspace{-1.5mm} & \vspace{-1.5mm}\\}
\newcommand{\tw}[1]{{}^#1\!}
\newcommand{\Irrg}[1]{\Irr_{p',\sigma}(#1)}
\renewcommand{\mod}{\bmod \,}

\marginparsep-0.5cm

\renewcommand{\thefootnote}{\fnsymbol{footnote}}
\footnotesep6.5pt

\thanks{This work was funded by the European Union - Next Generation EU, Missione 4 Componente 1, PRIN 2022- 2022PSTWLB - Group Theory and Applications, CUP B53D23009410006 as well as the Spanish Ministerio de Ciencia e Innovaci\'on (Grant  PID2022-137612NB-I00 funded by MCIN/AEI/ 10.13039/501100011033 and “ERDF A way of making Europe”). }

\title{Character degrees and local subgroups revisited}
\author{J. Miquel Mart\'inez}
%\date{}

\maketitle
\begin{abstract}{Let $p$ and $q$ be different primes and let $G$ be a finite $q$-solvable group. We prove that $\Irr_{p'}(G)\sbs\Irr_{q'}(G)$ if and only if $\norm G P\sbs\norm G Q$ and $\cent {Q'}P=1$ for some $P\in\Syl_p(G)$ and $Q\in\Syl_q(G)$. Further, if $B$ is a $q$-block of $G$ and $p$ does not divide the degree of any character in $\Irr(B)$ then we prove that a Sylow $p$-subgroup of $G$ is normalized by a defect group of $B$. This removes the $p$-solvability condition of two theorems of G. Navarro and T. R. Wolf.}
\end{abstract}

\section{Introduction}

Let $G$ be a finite group, and let $p$ and $q$ be different primes. We denote by $\Irr_{p'}(G)$ the set of irreducible characters of $G$ whose degree is not divisible by $p$. In this note we prove the following result.

\begin{thmAcl}
Let $G$ be a $q$-solvable group. Then $\Irr_{p'}(G)\sbs\Irr_{q'}(G)$ if and only if $\norm G P\sbs\norm G Q$ and $\cent {Q'}P=1$ for some $P\in\Syl_p(G)$ and $Q\in\Syl_q(G)$.
\end{thmAcl}

Theorem A was proven for $\{p, q\}$-separable groups in \cite{Nav-Wol98} (by Burnside's $p^aq^b$ theorem, $\{p,q\}$-separable groups are exactly the groups that are both $p$-solvable and $q$-solvable). The authors point out that the $q$-solvability hypothesis is necessary (as in \cite[Example 10]{Nav-Wol98}, take $K=\PSL(2,3^5)$ and $a$ an automorphism of $K$ of order $5$, then $G=K\langle a \rangle$ with $p=5$ and $q=11$ shows that Theorem A does not hold for $p$-solvable groups). They speculate that the stronger version for $q$-solvable groups ``\emph{is heavily related to the validity of McKay's conjecture
and similar results}''. Here we confirm their speculation by using the recently proved McKay conjecture \cite{Cab-Spa}.

The main result of \cite{Bel-Nav00} gives a version of \cite{Nav-Wol98} for $q$-Brauer characters in $\{p, q\}$-separable groups. This was later relaxed to $q$-solvable groups in \cite{Bon-Nav-Riz-San22}, and we use this latter result to obtain Theorem A. We note that, unlike the results for $\{p, q\}$-separable groups, both \cite[Theorem A]{Bon-Nav-Riz-San22} and our Theorem A rely on the classification of finite simple groups.

We mention that, without any separability condition, it has been proven in \cite{Nav-Tie} that $\Irr_{p'}(G)=\Irr_{q'}(G)$ if and only if $\norm G Q=\norm G P$ for some $Q\in\Syl_q(G)$ and $P\in\Syl_p(G)$ and both $P$ and $Q$ are abelian. 

Focusing on blocks, we prove the following.

\begin{thmBcl}\label{thm:B}
Let $G$ be a $q$-solvable group and let $B$ be a $q$-block of $G$ with defect group $Q$. If $p$ does not divide $\chi(1)$ for all $\chi\in\Irr(B)$ then $Q\sbs\norm G P$ for some $P\in\Syl_p(G)$.
\end{thmBcl}

The statement of Theorem B was proven to hold for $\{p, q\}$-separable groups in \cite{Nav-Wol01} without using the classification, but our proof of Theorem B uses the classification. As mentioned in the introduction of \cite{Nav-Wol01}, the principal $2$-block of the sporadic Janko group $J_1$ for $p=2$ and $q=5$ shows that the conclusion of Theorem B does not hold for arbitrary finite groups, and requires some separability condition (although it has been proven recently in \cite[Theorem C]{Mor-Sch24} that there are remarkably few counterexamples for principal blocks of arbitrary finite groups). We have not found any $p$-solvable counterexamples to the conclusion of Theorem B for nonprincipal blocks.

We point out that, while Theorem A depends on the McKay conjecture being true, Theorem B depends on the McKay conjecture being proved \emph{in a certain way}, that is, proving that the inductive McKay conditions defined in \cite{Isa-Mal-Nav07} hold for certain simple groups (we do not need this explicitly here but it is used in the proof of \cite[Theorem A]{LNT}, on which our work relies). 

The divisibility of degrees of characters in a $q$-block by a different prime $p$ has become an active research topic in recent years \cite{Gia-Mal-Val19, Mal-Nav20, Wil, Gia-Mec24}. It remains open to describe group-theoretically when $p$ does not divide the degrees of all height zero characters in a $q$-block. For $\{p,q\}$-separable groups, and after Fong--Reynolds reductions such as the ones in the proof of Theorem B, this would already be a relative version of \cite[Theorem A]{Nav-Riz-San22} with $\pi=\{p\}'$ and $\rho=\{p, q\}'$ (by relative, we mean that the condition on degrees is satisfied only for the irreducible characters of $G$ lying over a character of a normal subgroup).

\subsection*{Acknowledgments} 

The author thanks Alexander Moret\'o, Gabriel Navarro, Noelia Rizo and Damiano Rossi for helpful conversations on the topic of this note. Furthermore, the author thanks Martin Liebeck, Cheryl Praeger and Pham Huu Tiep for communicating the results from \cite{LNT}. The author also wishes to thank the referee for their many useful comments and suggestions.

\section{Theorem A}

Theorem A is proven in two steps. First we assume the local condition $\norm G P\sbs \norm G Q$ and prove the result in this case. To complete the proof of Theorem A we use the Fong--Swan theorem and the main result of \cite{Bon-Nav-Riz-San22}. One of our main ingredients is the following refinement of the McKay conjecture for $q$-solvable groups, which relies on the classification of finite simple groups.

\begin{thm}\label{thm:rizo}
Let $G$ be a $q$-solvable group and $Q\in\Syl_q(G)$. Then there is a bijection
$$f:\Irr_{q'}(G)\rightarrow\Irr_{q'}(\norm G Q)$$
where $f(\chi)(1)$ divides $\chi(1)$ and $\chi(1)/f(\chi)(1)$ divides $|G:\norm G Q|$ for all $\chi\in\Irr_{q'}(G)$.
\end{thm}
\begin{proof}
This is proved for solvable groups in  \cite{Tur07} and for $q$-solvable groups in \cite{Riz19} assuming the divisibility of degrees in the Glauberman correspondence, which was proved in \cite{Gec20} after a reduction to simple groups in \cite{Har-Tur94}.
\end{proof}

In this paper we write $\pi=\{p,q\}$, so $\Irr_{\pi'}(G)$ denotes the set of irreducible characters of $G$ whose degree is not divisible by $p$ nor $q$. 

\begin{pro}\label{pro}
Let $G$ be a $q$-solvable group, $P\in \Syl_p(G)$, $Q\in\Syl_q(G)$ and assume $\norm G P\sbs\norm G Q$. Then $\cent {Q'}P=1$ if and only if $\Irr_{p'}(G)\sbs\Irr_{q'}(G)$.
\end{pro}

\begin{proof}
First we claim that $$\Irr_{p'}(G)\sbs\Irr_{q'}(G)$$ if and only if $$\Irr_{p'}(\norm G Q)\sbs\Irr_{q'}(\norm G Q).$$ Let $f:\Irr_{q'}(G)\rightarrow \Irr_{q'}(\norm G Q)$ be the bijection from Theorem \ref{thm:rizo}. By hypothesis, $p$ does not divide $|G:\norm G Q|$ and therefore $\chi\in\Irr_{q'}(G)$ has $p'$-degree if and only if $f(\chi)\in\Irr_{q'}(\norm G Q)$ has $p'$-degree. Thus $|\Irr_{\pi'}(G)|=|\Irr_{\pi'}(\norm G Q)|$. Now, by the McKay conjecture \cite{Cab-Spa}, $|\Irr_{p'}(G)|=|\Irr_{p'}(\norm G P)|=|\Irr_{p'}(\norm G Q)|$ because $\norm G P=\norm {\norm G Q} P$. Thus $\Irr_{p'}(G)\sbs\Irr_{q'}(G)$ if and only if $\Irr_{p'}(G)=\Irr_{\pi'}(G)$ if and only if  $\Irr_{p'}(\norm G Q)=\Irr_{\pi'}(\norm G Q)$ if and only if  $\Irr_{p'}(\norm G Q)\sbs\Irr_{q'}(\norm G Q)$, as desired. 

It follows that we may assume $Q\normal G$. It is easy to see that in this case $\Irr_{q'}(G)=\Irr(G/Q')$.  By the Glauberman correspondence \cite[Theorem 2.9]{Nav18}, $\cent{Q'}P=1$ if and only if $|\Irr_P(Q')|=1$, where $\Irr_P(Q')$ denotes the set of $P$-invariant characters of $Q'$.

Assume first that $\Irr_{p'}(G)\sbs\Irr_{q'}(G)$. Let $\tau\in\Irr_P(Q')$. We have that $\tau$ extends to $\mu\in\Irr(PQ')$ by \cite[Corollary 6.2]{Nav18}. Now $\mu^{G}(1)=\tau(1)|G:PQ'|$ is not divisible by $p$ and it follows that there is $\chi\in\Irr_{p'}(G)$ with $[\chi,\mu^G]\neq 0$, so $\chi_{Q'}$ contains $\tau$. By hypothesis, $\chi$ has $q'$-degree, so $Q'\sbs\ker\chi$ and therefore $\tau=1_{Q'}$. It follows that $|\Irr_{P}(Q')|=1$ and we are done.

Now assume $\cent {Q'}P=1$ so that $\Irr_P(Q')=1$. Then if $\chi\in\Irr_{p'}(G)$, by \cite[Lemma 9.3]{Nav18}, $\chi_{Q'}$ contains a $P$-invariant irreducible constituent. By hypothesis, $\chi_{Q'}$ contains $1_{Q'}$ so $\chi\in\Irr(G/Q')$ and we conclude that $\chi\in\Irr_{q'}(G)$, as desired.
\end{proof}

\textit{Proof of Theorem A.} By Proposition \ref{pro}, it suffices to show that $\Irr_{p'}(G)\sbs\Irr_{q'}(G)$ implies $\norm G P\sbs\norm G Q$ for some $P\in\Syl_p(G)$ and $Q\in\Syl_q(G)$. Now let $\varphi\in\IBr_{p'}(G)$ (where we are considering $q$-Brauer characters). By the Fong--Swan theorem \cite[Theorem 10.1]{Nav98} there is $\chi\in\Irr_{p'}(G)$ with $\chi^0=\varphi$. By hypothesis, $\chi\in\Irr_{q'}(G)$ and therefore $\varphi\in\IBr_{q'}(G)$. It follows that $\IBr_{p'}(G)\sbs\IBr_{q'}(G)$ so by \cite[Theorem A]{Bon-Nav-Riz-San22} we have $\norm G P\sbs\norm G Q$ for some $P\in\Syl_p(G)$ and $Q\in\Syl_q(G)$, and the result follows. \qed

\section{Theorem B}

This section is devoted to the proof of Theorem B. Our strategy is to prove a version of \cite[Theorem 3]{Nav-Wol01} and then apply Fong--Reynolds reductions. Our proof relies on the following result due to M. W. Liebeck, G. Navarro, C. E. Praeger and P. H. Tiep, which appeared in \cite{LNT}, and which uses the classification of finite simple groups.

\begin{thm}[Liebeck--Navarro--Praeger--Tiep]\label{lnpt}
Let $G$ be a finite group and $\pi$ a set of primes. Let $Z\normal G$ and $\lambda\in\Irr(Z)$. Assume that $\chi(1)/\lambda(1)$ is a $\pi'$-number for any $\chi\in\Irr(G|\lambda)$. If $\{5,7\}\not\subseteq \pi$ or $\{2,3\}\not\subseteq\pi'$ or $G$ does not involve $\AAA_7$ then $G/Z$ has abelian Hall $\pi$-subgroups.
\end{thm}

The following is our version of \cite[Theorem 3.1]{Nav-Wol01}.

\begin{thm}\label{thm:stronger result 2}
Let $G$ be a finite group, let $Z\normal G$ and $\lambda\in\Irr(Z)$ be $G$-invariant. Assume that $G/Z$ is $q$-solvable and that $\lambda$ extends to $Q$ where $Q/Z\in\Syl_q(G/Z)$. If $p$ does not divide $\chi(1)/\lambda(1)$ for all $\chi\in\Irr(G|\lambda)$ then a Sylow $q$-subgroup of $G/Z$ normalizes a Sylow $p$-subgroup of $G/Z$.
\end{thm}

\begin{proof}
Write $\pi=\{p,q\}$. We argue by induction on $|G/Z|$.

\medskip

\textit{Step 1: We may assume $Z\sbs \zent G$, $Z$ has $\pi'$-order and $\lambda$ is faithful. In particular, $G$ is not $p$-solvable.}

\medskip

By \cite[Corollary 5.9]{Nav18} there is an isomorphic character triple $(G^*, Z^*, \lambda^*)$ where $Z^*$ is central and $\lambda^*$ is faithful. Now if there is a Sylow $q$-subgroup $Q^*/Z^*$ normalizes a Sylow $p$-subgroup $P^*/Z^*$ of $G^*/Z^*$ then $Q/Z$ normalizes $P/Z$ via the isomorphism $G^*/Z^*\rightarrow G/Z$, so we may assume $Z$ is cyclic and central and $\lambda$ is faithful.

Now write $\lambda=\lambda_p\lambda_{p'}$ where $\lambda_p$ has $p$-power order and $\lambda_{p'}$ has order coprime to $p$. Let $\chi\in\Irr(G|\lambda)$ and $P/Z\in\Syl_p(G/Z)$. Since $\chi$ has $p'$-degree, $\chi_P$ contains a linear character by \cite[Theorem 5.12]{Nav18}, which must lie over $\lambda$. Therefore $\lambda$ extends to $P$. Since $\lambda_p$ is a power of $\lambda$ in the group $\Irr(Z)$, $\lambda_p$ also extends to $P$ and by \cite[Corollary 6.2]{Nav18}, $\lambda_p$ extends to $\hat\lambda_p\in\Irr(G)$. By \cite[Problem 5.2]{Nav18}, $(G,Z,\lambda)$ is isomorphic to $(G, Z,(\hat\lambda_p)_Z\lambda)=(G,Z,\lambda_{p'})$ so it is no loss to assume $\lambda$ has $p'$-order. Since $\lambda$ is faithful and $Z$ is central, this implies that $Z$ has $p'$-order. By writing $\lambda=\lambda_q\lambda_{q'}$ and arguing as before by for the prime $q$ we conclude that we can assume that $Z$ also has $q'$-order (recall that $\lambda$ extends to $Q$ by hypothesis).

Since $Z$ has $p'$-order it is clear that $\lambda$ is $p'$-special in the sense of \cite[Section 2]{Nav-Wol01}. Using \cite[Theorem 3.1]{Nav-Wol01} we may assume $G$ is not $p$-solvable.

\medskip

\textit{Step 2: We may assume $\oh{q'}{G/Z}=1$ and therefore $\oh{\pi'}{G/Z}=1$. In particular, $Z=\oh{q'}{G}=\oh{\pi'}G$.}

\medskip

Write $M/Z=\oh{q'}{G/Z}$ and assume $M>Z$.
By \cite[Lemma 2.1]{Nav-Wol01} (applied to $\pi=\{q\}'$), there is a $Q$-invariant $\mu\in\Irr(M|\lambda)$. Since $M$ is a $q'$-group, $\mu$ extends to $Q$ by \cite[Corollary 6.2]{Nav18}. Furthermore, $\Irr(G|\mu)\sbs\Irr(G|\lambda)$ and therefore every character in $\Irr(G|\mu)$ has $p'$-degree. It follows from the Clifford correspondence that $G_\mu$ contains a Sylow $p$-subgroup $P$ of $G$. By induction, there is some Sylow $q$-subgroup $Q_0\in\Syl_q(G)$ such that $Q_0M/M$ normalizes $PM/M$ and therefore $Q_0$ normalizes $PM$. Therefore $Q_0$ acts coprimely on $PM$ and by \cite[18.7(1)]{Asch}, $Q_0$ normalizes a Sylow $p$-subgroup of $PM$, which is a Sylow $p$-subgroup of $G$. Thus we may assume $M=Z$, and since $\oh{\pi'}{G/Z}\sbs\oh{q'}{G/Z}$, the claim of Step 2 follows.

\medskip

\textit{Step 3: We may assume $G$ has a normal Sylow $q$-subgroup $Q$.}

\medskip

Let $L=\oh{q}G$. We have that $L>1$ because $\oh{q'}G=Z$ is central. Since $ZL=Z\times L$ we have $\Irr(G|\lambda\times 1_L)\sbs\Irr(G|\lambda)$ can be identified with $\Irr(G/L|\overline{\lambda})$ (where $\overline{\lambda}\in\Irr(ZL/L)$ is the character naturally associated to $\lambda$ via the natural isomorphism $Z\cong ZL/L$). Now $ZQ=Z\times Q$ so $\lambda\times 1_Q\in\Irr(ZQ)$ extends $\lambda\times 1_L$ and therefore $\overline{\lambda}$ extends to $ZQ/L\in\Syl_q(G/L)$. By induction we have that there is $P\in\Syl_p(G)$ such that $PL/L$ is normalized by a Sylow $q$-subgroup of $G/L$. Write $R/L=\oh{q'}{G/L}$ and $T/R=\oh{q}{G/R}$. It follows that $T/R$ normalizes $PR/R\in\Syl_p(G/R)$ and therefore $PR/R\sbs\cent{G/R}{T/R}$. Since $\oh{q'}{G/R}=1$, Hall--Higman 1.2.3 \cite[Theorem 3.21]{Isa08} implies that $\cent{G/R}{T/R}\sbs T/R$, so $P\sbs R$.

 Assume $R<G$. Since every $\psi\in\Irr(R|\lambda)$ lies below some $\chi\in\Irr(G|\lambda)$ it follows that every $\psi\in\Irr(R|\lambda)$ has $p'$-degree. By induction, a Sylow $q$-subgroup $R/Z$ normalizes $P/Z$. In particular, $LZ/Z$ normalizes $P/Z$ and therefore $P/Z\sbs\cent{G/Z}{L/Z}\sbs LZ/Z$. It follows that $P=1$ and therefore $G$ is a $p'$-group and the result follows in this case. Therefore we may assume $R=G$ and then $G$ has a normal Sylow $q$-subgroup $Q$.
 
\medskip

\textit{Step 4: We may assume $Q$ is elementary abelian}

\medskip

Notice that $Z\Phi(Q)=Z\times\Phi(Q)$ and that $\Irr(G|\lambda\times 1_{\Phi(Q)})\sbs\Irr(G|\lambda)$. For any subgroup $X\leq G$ write $\overline{X}=X\Phi(Q)/\Phi(Q)$. We may identify $\Irr(G|\lambda\times 1_{\Phi(Q)})$ with $\Irr(\overline{G}|\overline{\lambda})$ where $\overline{\lambda}\in\Irr(\overline{Z})$ is the character canonically associated to $\lambda$ via the natural isomorphism $Z\cong \overline{Z}$, and therefore $p$ does not divide the degree of any character in $\Irr(\overline{G}|\overline{\lambda})$. Arguing as before, $\overline{\lambda}$ extends to $\overline{Q}$. If $\Phi(Q)>1$ then, by induction, $\overline{Q}\sbs\norm{\overline{G}}{\overline{P}}$ for some $P\in\Syl_p(G)$, which implies that $Q\sbs\Phi(Q)\norm G P$. If $R=Q\cap\norm G P\in\Syl_q(\norm G P)$ then $Q=\Phi(Q)R$ and by \cite[5.2.3]{Kur-Ste04} this implies $Q=R$, so $Q\sbs\norm G P$. Thus we may assume $\Phi(Q)=1$ so $Q$ is elementary abelian.

\medskip

\textit{Final step}

\medskip

By It\^o's theorem, $q$ does not divide the degree of any irreducible character in $\Irr(G|\lambda)$ and therefore every $\chi\in\Irr(G|\lambda)$ has $\pi'$-degree. Since $G$ is $q$-solvable, if $q$ is any prime divisor of $|\mathsf{A}_7|$ then $G$ does not involve $\mathsf{A}_7$. Therefore we are under the hypotheses of Theorem \ref{lnpt} and we have that $G/Z$ has abelian Hall $\pi$-subgroups, so $QZ/Z$ normalizes a Sylow $p$-subgroup of $G/Z$, as desired.
\end{proof}

\textit{Proof of Theorem B.} Argue by induction on $|G|$. We work to show that we may assume $B$ covers some $\lambda\in\Irr(\oh{q'}G)$ that is $G$-invariant. Assume otherwise and let $\theta\in\Irr(\oh{q'}G)$ be $(B,Q)$-good for a defect group $Q$ of $B$, in the sense of \cite[Proposition 1.3]{Mar-Ros23}. Let $U, \nu$ be as in \cite[Proposition 1.3]{Mar-Ros23}. Then induction defines a bijection $\Irr(U|\nu)\rightarrow\Irr(B)$ and by \cite[Theorem 10.20]{Nav98} there is a block $B_U$ of $U$ with $\Irr(U|\nu)=\Irr(B_U)$ and the defect groups of $B_U$ are defect groups of $B$ (again by \cite[Propostion 1.3]{Mar-Ros23}). Since $\Irr(B)=\{\psi^G\mid\psi\in\Irr(B_U)\}$ and every character in $\Irr(B)$ has $p'$-degree, it is straightforward that every character in $\Irr(B_U)$ has degree coprime to $p$, and that $U$ contains a Sylow $p$-subgroup $P$ of $G$. By induction, some defect group of $B_U$ normalizes $P$. Since the defect groups of $B_U$ are defect groups of $B$, we are done in this case.

Therefore we may assume that there is a $G$-invariant $\lambda\in\Irr(\oh{q'}G)$ covered by $B$, so by \cite[Theorem 10.20]{Nav98} $B$ has maximal defect and $\Irr(B)=\Irr(G|\lambda)$. Notice that by \cite[Corollary 6.2]{Nav18}, $\lambda$ extends to $Q$ where $Q/\oh{q'}G\in\Syl_q(G/\oh{q'}G)$. By Theorem \ref{thm:stronger result 2} we obtain that $Q/\oh{q'}G$ normalizes a Sylow $p$-subgroup $P/\oh{q'}G$ of $G/\oh{q'}G$. If $R$ is a Sylow $q$-subgroup of $Q$ then $R$ normalizes $P/\oh{q'}G$. Now $R$ is a $q$-group acting coprimely on $P$ so it normalizes a Sylow $p$-subgroup of $P$ by \cite[18.7(1)]{Asch}. Since $P$ contains a Sylow $p$-subgroup of $G$, we are done.
\qed

\bigskip

It remains open to characterize group theoretically when $p$ does not divide the degrees of all (or only all height zero) characters in a $q$-block. The main result of \cite{Gia-Mal-Val19} characterizes finite groups whose principal $q$-block (for odd $q$) contains no characters of even degree, modulo the sporadic group $M_{22}$ not appearing as a composition factor if $q=7$. 
It has been conjectured recently that for all primes $p\neq q$ dividing $|G|$, $p$ does not divide the degrees of the height zero characters in the principal $q$-block if and only if $G$ has a normal Sylow $q$-subgroup \cite[Conjecture 6.1]{Mor-Sch24}. 
For arbitrary blocks, it is harder to make predictions.

\medskip

\bibliographystyle{alpha}

{\sc Departament de Matem\`atiques, Universitat de Val\`encia, 46100
Burjassot, Val\`encia.}

\textit{Email address:} \href{mailto:josep.m.martinez@uv.es}{\texttt{josep.m.martinez@uv.es}}

\end{document}